\documentclass[11pt,reqno]{amsart}

\usepackage{geometry}
\usepackage{lmodern}
\usepackage[T1]{fontenc}

\usepackage{amsmath}
\usepackage{amssymb}
\usepackage{amsthm}
\usepackage{mathtools}
\usepackage{tikz-cd}

\usepackage{comment}
\usepackage{enumitem} 

\usepackage[colorlinks=true, linkcolor=blue, citecolor=magenta, urlcolor=blue]{hyperref}
\usepackage{aliascnt} 
\usepackage[capitalize,nameinlink]{cleveref}

\theoremstyle{plain} 
\newtheorem{theorem}{Theorem}[section]

\newaliascnt{lemma}{theorem}
\newtheorem{lemma}[lemma]{Lemma}
\aliascntresetthe{lemma}
\crefname{lemma}{lemma}{lemmas}
\Crefname{lemma}{Lemma}{Lemmas}

\newaliascnt{corollary}{theorem}
\newtheorem{corollary}[corollary]{Corollary}
\aliascntresetthe{corollary}
\crefname{corollary}{corollary}{corollaries}
\Crefname{corollary}{Corollary}{Corollaries}

\newaliascnt{claim}{theorem}

\aliascntresetthe{claim}
\crefname{claim}{claim}{claims}
\Crefname{claim}{Claim}{Claims}

\theoremstyle{definition} 

\newaliascnt{definition}{theorem}

\aliascntresetthe{definition}
\crefname{definition}{definition}{definitions}
\Crefname{definition}{Definition}{Definitions}

\newaliascnt{example}{theorem}

\aliascntresetthe{example}
\crefname{example}{example}{examples}
\Crefname{example}{Example}{Examples}

\theoremstyle{remark} 

\newaliascnt{remark}{theorem}
\newtheorem{remark}[remark]{Remark}
\aliascntresetthe{remark}
\crefname{remark}{remark}{remarks}
\Crefname{remark}{Remark}{Remarks}
\begin{document}

\title{On the Super-Schottky Ideal and Powers of the Classical Schottky Ideal for $g \geq 5$}

\author{Yuanyuan Shen}

\begin{abstract}
The super Schottky ideal defines the image of the superperiod map, which takes a family of genus $g$ supercurves to its Jacobian. We prove that if $g \geq 5$, then $d=g$ is the minimal number such that $\mathcal{I}_{Sch,g}^d \subset \mathcal{I}_{s-Sch,g}$, where $\mathcal{I}_{Sch,g}$ is the Schottky ideal and $\mathcal{I}_{s-Sch,g}$ is the super Schottky ideal. 
\end{abstract}

\maketitle

\section{Introduction}

\noindent We set the background for studying super Riemann surfaces \cite{Witten2012superRiemann}. A supercurve is a compact, connected $(1|1)$ dimensional complex supermanifold, and a family of super Riemann surfaces is a morphism of supercurves $\pi: X \to S$ together with a rank $(0|1)$ subbundle $\mathcal{D}$ of the relative tangent bundle $T_{X/S}$. Let $\mathfrak{M}_g$ be the moduli space parametrizing families of genus $g$ super Riemann surfaces, and let $\mathfrak M_g^+$ be its
component corresponding to even spin structures. Then the reduced space of $\mathfrak{M}^+_g$ can be identified with $S_g^+$, the moduli space of genus $g$ spin curves $(C,L)$. Over the locus $\mathcal{U}_g \subset \mathfrak{M}^+_g$ of curves with theta characteristics of vanishing cohomology, we define the superperiod map $per: \mathcal{U}_g \to LG_{2g}$ from $\mathcal{U}_g$ to the Lagrangian Grassmannian of the $2g$-dimensional symplectic vector space over $\mathbb{C}$.\\

\noindent Let $\mathcal{I}_{Sch,g}$ define the image of the ordinary period map and $\mathcal{I}_{s-Sch,g}$ define the image of the superperiod map, we prove the following result. \\

\begin{theorem}
    Let d be the minimal number such that $\mathcal{I}_{Sch,g}^d \subset \mathcal{I}_{s-Sch,g}$. Then $d = g$ for all $g\geq 5$. 
\end{theorem}

\noindent The problem has previously been researched by Felder, Kazhdan and Polishchuk \cite{FelderKazhdanPolishchuk2019}. For $g=2,3$, the classical Schottky ideal is zero, and $I_{\mathrm{Sch},g}=I_{s-\mathrm{Sch},g}$. For $g \geq 4$, indeed $\mathcal{I}_{Sch,g}^g \subset \mathcal{I}_{s-Sch,g}$ since $\mathfrak{M}^+_g$ has $2g-2$ odd variables. The case for $g = 4$ is established by the companion paper \cite{DonagiNoja-genus4}. For $g \geq 5$, Felder, Kazhdan and Polishchuk proved that the minimal $d$ such that $\mathcal{I}_{Sch,g}^d \subset \mathcal{I}_{s-Sch,g}$ is $d = g$ when $g$ is odd and $d \geq g-1$ when $g$ is even by expressing the codifferential of the superperiod map as a Massey product and evaluating it at the hyperelliptic limits. This paper improves the bound by proving the minimal $d$ such that $\mathcal{I}_{Sch,g}^d \subset \mathcal{I}_{s-Sch,g}$ is $d = g$ when $g$ is even.\\

\noindent This paper will start with the framework established by the companion paper \cite{DonagiNoja-genus4}. Given a curve $C$, consider the line bundle on $C \times C$
\begin{align*}
    \mathcal{O}(a,b,c) := p_1^*K_C^{\otimes a/2} \otimes p_2^*K_C^{\otimes b/2} \otimes \mathcal{O}_{C \times C}(c\Delta),
\end{align*}

\noindent where $p_1, p_2 : C \times C \to C$ are the projections, and $\Delta$ is the diagonal. Restriction to the diagonal gives the short exact sequence:
\begin{equation*}
0 \to \mathcal{O}(a,b,c-1) \to \mathcal{O}(a,b,c) \xrightarrow{\text{Res}} (K_C)^{\otimes(a+b-2c)/2} \to 0
\end{equation*}
\noindent In particular, when $a = b$, under the involution $(x,y) \to (y,x)$ the exact sequence decomposes into symmetric and antisymmetric parts. Let $C$ be such that $H^0(K^{\frac{1}{2}}) = 0$, consider the following diagram of sheaves on $C$ induced by the above exact sequence


\begin{equation*}
\begin{tikzcd}
0 \arrow[r] & H^0(\mathcal{O}(2,2,-1)^+) \arrow[r] \arrow[d, "s_1"] & H^0(\mathcal{O}(2,2,0)^+) \arrow[r] \arrow[d, "s_2"] & H^0(K^2) \arrow[r] \arrow[d, "s_3"] & 0 \\
0 \arrow[r] & H^0(\mathcal{O}(3,3,0)^+) \arrow[r] & H^0(\mathcal{O}(3,3,1)^+) \arrow[r] & H^0(K^2) \arrow[r] & 0,
\end{tikzcd}
\end{equation*}

\noindent Following the convention in \cite{DonagiWitten2013}, $H^0(\mathcal{O}(2,2,-1)^+)$ denotes its symmetric part and $ H^0(\mathcal{O}(3,3,0)^+) $ denotes its antisymmetric part. The vertical maps $s_i$ are multiplication by the unique Szegő-kernel section whose residue along the diagonal is 1. In particular, we identify $s_1$ with the codifferential of the superperiod map. \\

\noindent Consider a family of curves $\{ (C_t,L_t)\}$ with $H^0(K_{C_t}^{\frac{1}{2}}) = 0$ specializing to a spin curve $(C_0,L_0)$ such that $H^0(K_{C_0}^{\frac{1}{2}}) = 2$, and let each $C_t$ correspond to codifferential $s_{1,t}$ for $t \neq 0$. At $t = 0$, the codifferential blows up; however, the product $t s_{1,t}$ has a well defined limit. Consequently, we introduce the rescaled codifferential $s'_{1,t} := t s_{1,t}$ such that its image has the same rank as the image of $s_{1,t}$, and it extends to a well defined limit $s'_{1,0} = \lim_{t \to 0} s'_{1,t}$ at $t = 0$. \\

\noindent At the limiting spin curve $(C_0,L_0)$, the limiting codifferential $s'_{1,0}$ can be computed explicitly. We prove that it is possible to pick quadrics $Q_{null}$, the degenerate cone containing $C_0$, and $\tilde{Q}$, another quadric containing $C_0$, such that the image of $Q_{null} + \epsilon \tilde{Q}$ under $s'_{1,0}$ has full rank for sufficiently small $\epsilon$. Then the image of $s_{1,t}$ has full rank for sufficiently small $t$. Taking $f$ to be the lift of $Q_{null} + \epsilon \tilde{Q}$ in $\mathcal{I}_{Sch}$, we have $per^* f^{g-1} \neq 0$ and $f^{g-1} \not\in  \mathcal{I}_{s-Sch,g}$.

\vspace{1em}

\section{The superperiod map and its second variation}

\subsection{Supercurves and the superperiod map}
\noindent We define the superperiod map. Let $\pi: X \to S$ be a smooth proper supercurve over a superscheme $S$, then its superconformal structure $\mathcal{D}$ is a rank $(0|1)$ maximally non-integrable sub-bundle of the tangent bundle $T_{X/S}$. By maximally non-integrable, we refer to the property that $D^2 = \frac{1}{2}\{D,D\}$ is nowhere proportional to $D$ for any local non-zero section $D$ of $\mathcal{D}$, therefore the projection to $T_{X/S} / \mathcal{D}$ gives an isomorphism $\mathcal{D} \otimes \mathcal{D} \cong T_{X/S} / \mathcal{D}$, inducing the short exact sequence $0 \to \mathcal{D} \to T_{X/S} \to \mathcal{D}^{\otimes 2} \to 0$ and the dual sequence $0 \to (\mathcal{D}^{\vee})^{\otimes 2} \to \Omega_{X/S} \to \mathcal{D}^{\vee} \to 0 $. \\

\noindent By taking the Berezinian of the above sequence, we show $\mathcal{D}^{\vee} \cong \omega_{X/S} \coloneqq Ber(\Omega_{X/S} )$ and by composing with the exterior $d: \mathcal{O}_X \to \Omega_{X/S}$, we
obtain the $\mathcal{O}_S$-linear derivation $\delta: \mathcal{O}_X \to  \omega_{X/S}$. Specifically, note that $\delta$ is $\mathcal{O}_S$-linear since functions on $S$ vanishes in the exterior, but is not $\mathcal{O}_X$-linear. This induces the short exact sequence $0 \to \mathbb{C}_{X/S} \to \mathcal{O}_X \to \omega_{X/S} \to 0$ where $\mathbb{C}_{X/S} = \pi^{-1} \mathcal{O}_S$ are functions locally constant along fibres of $\pi$. Taking the right derived functor gives morphism $\pi_* \omega_{X/S} \to R^1 \pi_* \mathbb{C}_{X/S} \cong R^1 \pi_* \mathbb{Z} \otimes \mathcal{O}_S$.  \\

\noindent Now let $\mathfrak{M}^+_g$ be the even component of the moduli space of genus $g$ supercurves, and $\mathcal{U}_g \subset \mathfrak{M}^+_g$ be the locus corresponding to curves with theta characteristics of vanishing cohomology. Furthermore, consider the universal curve of genus $g$ supercurves, $\mathfrak{X}_g \to \mathfrak{M}^+_g$, and its restriction morphism $\hat{\pi}: X_g :=\mathfrak{X}_g \times_{\mathfrak{M}^+_g} \mathcal{U}_g \to  \mathcal{U}_g$ where $\hat{\pi}_*\omega_{X_g / \mathcal{U}_g}$ is locally free. Applying the above construction, we obtain the morphism 
\begin{align*}
    \hat{\pi}_*\omega_{X_g / \mathcal{U}_g} \to R^1 \pi_* \mathbb{Z} \otimes \mathcal{O}_{\mathcal{U}_g}
\end{align*}

\noindent Choosing locally on $\mathcal{U}_g$ a symplectic basis $\phi: R^1\pi_*\mathbb{Z} \cong \mathbb{Z}^{2g}$ determines a local map, $\hat{per}: \mathcal{U}_g \rightarrow \mathcal{H}_g \subset LG_{2g}$, where $\mathcal{H}_g$ is the Siegel upper half-space and $LG_{2g}$ is the Lagrangian Grassmannian of the $2g$-dimensional symplectic vector space over $\mathbb{C}$. Since a change of symplectic basis acts on $\mathcal{H}_g$ via the modular group $Sp(2g, \mathbb{Z})$, this local map naturally descends to a global superperiod map:$$per: \mathcal{U}_g \rightarrow \mathcal{A}_g \cong \mathcal{H}_g / Sp(2g, \mathbb{Z})$$

\noindent where $\mathcal{A}_g$ is moduli space of principally polarized abelian varieties.
\\

\subsection{The second variation map} Before examining the second variation map associated to the super period map, we first define the second variation map for a general supervariety. Let $X$ be a smooth supervariety, $\mathcal{N} \subset \mathcal{O}_X$ the nilradical, and $\Omega'_X \subset \Omega_X$ the $\mathcal{O}_X$-submodule generated by $dg$ where $g$ is an even function. Further let $Y$ be a reduced supervariety, and $\phi : X \rightarrow Y$ with restriction $\phi_{red} = \phi|_{X_{red}}$, then the following diagram commutes.

\[
\begin{tikzcd}
\mathbf{0} \arrow[r] & \mathbf{T}_{X_{red}} \arrow[r] \arrow[dr, "\mathbf{d}\phi_{red}"'] & (\Omega'_X|_{X_{red}})^\vee \arrow[r] \arrow[d] & (\mathcal{N}^2/\mathcal{N}^3)^\vee \arrow[r] & \mathbf{0} \\
 & & (\phi_{red})^*\mathbf{T}_Y & & 
\end{tikzcd}
\]

\noindent The unique map of $\mathcal{O}_{X_{red}}$-modules induced by the vertical arrow is the second variation map.

\[
d^{(2)}\phi : \wedge^2(\mathcal{N}/\mathcal{N}^2)^\vee \cong (\mathcal{N}^2/\mathcal{N}^3)^\vee \rightarrow coker(d\phi_{red})
\]

\noindent The codifferential is given by taking the dual of the second variation

\[
(d^{(2)}\phi)^\vee : (ker(d\phi^\vee_{red})) \rightarrow \wedge^2(\mathcal{N}/\mathcal{N}^2)
\]

\noindent Consider the superperiod map $per: \mathcal{U}_g \rightarrow \mathcal{A}_g $, we identify $\mathcal{U}_g$ with $X$, $\mathcal{A}_g $  with $Y$ and the odd cotangent space $H^0(K^{3/2})$ with $\mathcal{N}/\mathcal{N}^2$. Then the codifferential of the superperiod map is given by: 

\[
N_{\mathcal{M}_g,\mathcal{A}_g}^* \rightarrow \wedge^2 H^0(K^{3/2}).
\]

\subsection{The Szego kernel}\label{Szego}

Let $(C,L)$ be a spin curve where $L^{\otimes 2}\cong K_C$, and assume $H^0(C,L)=0$. Let $E(x,y)$ denote the prime
form and let $\theta_L$ be the theta function associated to the
theta characteristic $L$. The Szego kernel at $(C,L)$ is defined by
\[
S(x,y)
:=
\frac{
\theta_L\left(\int_y^x\omega\right)
}{
\theta_L(0)E(x,y)
}.
\]
where $\omega$ is a basis of holomorphic 1-forms. Since $H^0(C,L)=0$, the theta characteristic is nonsingular and
$\theta_L(0)\neq 0$. We note that the Szego kernel is holomorphic away from the diagonal, and along the diagonal it has a simple pole with residue $1$. By \cite{DHoker1989}, the codifferential of the superperiod map is multiplication by the corresponding Szego kernel.

\begin{lemma}
\label{prop:Szego-degeneration}
Let $(C_t,L_t)$ be a one-parameter family of even spin curves transverse to the theta-null divisor, with $\dim H^0(C_t,L_t)=0$ for $t \neq 0$ and $\dim H^0(C_0,L_0)=2$. Then the corresponding Szego kernels have a simple pole in
the parameter $t$. Consequently,
\[
S'_t:=tS_t
\]
extends holomorphically to $t=0$, and $S'_0\neq0$, $\operatorname{Res}_{\Delta_0}(S'_0)=0$.
\end{lemma}

\begin{proof}

By the companion paper \cite{DonagiNoja-genus4}.

\end{proof}

\section{Preliminaries}

\noindent We introduce the framework of sheaves on $C \times C$ \cite{DonagiWitten2013}. Let $\mathcal S_g^+$ be the moduli space of spin curves of genus $g$. We begin by considering a line bundle on $C \times C$:

\begin{align*}
    \mathcal{O}(a,b,c) := p_1^*K_C^{\otimes a/2} \otimes p_2^*K_C^{\otimes b/2} \otimes \mathcal{O}_{C \times C}(c\Delta),
\end{align*}

\noindent where $p_1, p_2 : C \times C \to C$ are the projections, and $\Delta$ is the diagonal. Restriction to the diagonal gives the short exact sequence:
\begin{equation*}
0 \to \mathcal{O}(a,b,c-1) \to \mathcal{O}(a,b,c) \xrightarrow{\text{Res}} (K_C)^{\otimes(a+b-2c)/2} \to 0
\end{equation*}
\noindent In particular, when $a = b$, we identify $p_1^*K_C^{\otimes a/2} \otimes p_2^*K_C^{\otimes b/2}$ with $p_2^*K_C^{\otimes b/2} \otimes p_1^*K_C^{\otimes a/2}$ with the choice of signs $dx^{\otimes a/2} dy^{\otimes b/2} \mapsto (-1)^{ab} dy^{\otimes b/2} dx^{\otimes a/2}$. Then the involution $(x,y) \to (y,x)$ on $ \mathcal{O}_{C \times C}(a,a,c)$ decomposes the above exact sequence into symmetric and antisymmetric parts, giving
\begin{equation}\label{exactsequence}
    0 \to \mathcal{O}(a, a, c - 1)^{(-)^{a-c}} \to \mathcal{O}(a, a, c)^{(-)^{a-c}} \xrightarrow{\text{Res}} K_C^{a-c} \to 0
\end{equation}
\begin{remark}
     Let a spin curve $(C,L)$ such that $H^0(L) = 0$, then its Szego kernel $S$ is in $H^0\left(\mathcal O(1,1,1)^+\right)$.
\end{remark}

\noindent We consider the following exact sequences of sheaves on a curve $C$.

\begin{lemma} \label{lemma31}
     Let $C$ be a non-hyperelliptic curve such that $H^0(K^{\frac{1}{2}}) = 0$. Then the following diagram is commutative

     \begin{equation*}
\begin{tikzcd}
0 \arrow[r] & H^0(\mathcal{O}(2,2,-1)^+) \arrow[r] \arrow[d, "s_1"] & H^0(\mathcal{O}(2,2,0)^+) \arrow[r] \arrow[d, "s_2"] & H^0(K^2) \arrow[r] \arrow[d, "s_3"] & 0 \\
0 \arrow[r] & H^0(\mathcal{O}(3,3,0)^+) \arrow[r] & H^0(\mathcal{O}(3,3,1)^+) \arrow[r] & H^0(K^2) \arrow[r] & 0,
\end{tikzcd}
\end{equation*}

\noindent where the horizontal exact sequences are induced by \cref{exactsequence}, and the maps $s_i$ are multiplication by the Szego kernel i.e. the unique section in $H^0(\mathcal{O}(1,1,1)^+)$ having residue $1$ along the diagonal.  Moreover, we can identify the top exact sequence with

\begin{align*}
    0 \to N_{\mathcal{M}_g,\mathcal{A}_g}^* \to T^* \mathcal{A}_g \to T^* \mathcal{M}_g \to 0
\end{align*}

\noindent and identify the map $s_1$ with codifferential of the superperiod map.

\end{lemma}
\begin{proof}

\noindent The top row is exact since the multiplication map is surjective for non-hyperelliptic curves. To show that the bottom row is exact, note that $H^1(\mathcal{O}_{C \times C}(3,3,0)) = H^1(K_C^{\frac{3}{2}}) \otimes H^0(K_C^{\frac{3}{2}}) \oplus H^0(K_C^{\frac{3}{2}}) \otimes H^1(K_C^{\frac{3}{2}}) = 0$ since $H^1(K_C^{\frac{3}{2}}) = 0$ by Serre duality.\\

\noindent To show the left square is commutative, note that the horizontal maps are inclusions, which commute with the vertical maps given by multiplications. To show the right square is commutative, consider a section in $H^0(O(2, 2, 0)^+)$ which is locally $f(x,y) dx \otimes dy$. Multiplying by the Szego kernel, which is locally $\left(\frac{1}{x-y} + h(x,y)\right) dx^{1/2} \otimes dy^{1/2}$ for some holomorphic function $h(x,y)$, followed by taking the residue along the diagonal gives $f(z,z) dz^2$. On the other hand, applying the restriction along the diagonal first gives $f(z,z) dz^2$, and applying the residue of the Szego kernel along the diagonal, which equals 1, gives $f(z,z) dz^2$. \\

\noindent To identify the top exact sequence, note that $ H^0(\mathcal{O}(2,2,0)^+) = T^* \mathcal{A}_g = Sym^2 H^0(K)$, $H^0(K^2)  = T^* \mathcal{M}_g$, and $H^0(\mathcal{O}(2,2,-1)^+) = N_{\mathcal{M}_g,\mathcal{A}_g}^* = I_2(C)$ where $I_2(C)$ is the ideal of quadrics vanishing on $C$. Moreover, both maps $H^0(\mathcal{O}(2,2,0)^+) \to H^0(K^2) $ and $T^* \mathcal{A}_g \to T^* \mathcal{M}_g $ are restrictions to the diagonal. To identify the map $s_1$, note that $H^0(\mathcal{O}(3,3,0)^+) = \wedge^2 H^0(K^{\frac{3}{2}})$, then $s_1$ is identified with the codifferential of the superperiod map by \cite{DHoker1989}.
\end{proof}

\begin{remark}
    Note that by convention in \cite{DonagiWitten2013}, $H^0(\mathcal{O}(2,2,0)^+)$ represents the symmetric part $\operatorname{Sym}^2 H^0(K)$  while $H^0(\mathcal{O}(3,3,0)^+)$ represents the antisymmetric part $\wedge^2 H^0(K^{\frac{3}{2}})$.
\end{remark}

\noindent Now take a family of spin curves $\{ (C_t,L_t)\}$ with $H^0(K_{C_t}^{\frac{1}{2}}) = 0$ transverse to the theta null divisor and specializing to a spin curve $(C_0,L_0)$ such that $H^0(K_{C_0}^{\frac{1}{2}}) = 2$. The Szego kernel is undefined at $(C_0,L_0)$. This motivates the rescaled Szego kernel $S'_t = tS_t$, which extends holomorphically to $t = 0$, with $S'_0 \neq 0$ by properties of the Szego kernel in \cref{Szego}.\\

\noindent Similarly consider the codifferential $s_{1,t}$ corresponding to $\{ (C_t,L_t)\}$ where $t \neq 0$, which blows up at $C_0$. Let $s'_{1,t} := t s_{1,t}$ be multiplication by $S'_t$, where its image has the same rank as the image of $s_{1,t}$ since rank is invariant under multiplication by a constant. Moreover, at $t = 0$ it extends to a well defined limit $s'_{1,0} = \lim_{t \to 0} s'_{1,t}$. Similarly let $s'_{2,0} = \lim_{t \to 0} s'_{2,t}$ and  $s'_{3,0} = \lim_{t \to 0} s'_{3,t}$. The maps induced by multiplication by $S'_t$ depend holomorphically on $t$, therefore the commutative diagrams from \cref{lemma31} for $t\neq0$ specialize to a commutative diagram at $t=0$.

\begin{equation}\label{commdia}
\begin{tikzcd}
0 \arrow[r] & H^0(\mathcal{O}(2,2,-1)^+) \arrow[r] \arrow[d, "s'_{1,0}"] & H^0(\mathcal{O}(2,2,0)^+) \arrow[r] \arrow[d, "s'_{2,0}"] & H^0(K^2) \arrow[r] \arrow[d, "s'_{3,0}"] & 0 \\
0 \arrow[r] & H^0(\mathcal{O}(3,3,0)^+) \arrow[r] & H^0(\mathcal{O}(3,3,1)^+) \arrow[r] & H^0(K^2) \arrow[r] & 0,
\end{tikzcd}
\end{equation}

\noindent We seek to prove that at $C_0$, we can explicitly calculate the codifferential map. We first prove the following lemma

\begin{lemma}\label{lem:injective}
Let $C$ be such that $H^0(K^{\frac{1}{2}}) = 2$. Given a short exact sequence 
\[
0 \to \mathcal{O}(1,1,0)^+ \to \mathcal{O}(1,1,1)^+ \to \mathcal{O} \to 0,
\]
show that the connecting homomorphism $\delta: H^0(\mathcal{O}) \to H^1(\mathcal{O}(1,1,0)^+)$ is injective.
\end{lemma} 

\begin{proof}

\noindent Consider $1 \in H^0(\mathcal{O})$, its local lift in $H^0(\mathcal{O}(1,1,1)^+)$ on an open set $\mathcal{U}_i$ has the form

\[
\left( \frac{1}{x-y} + R_i(x,y) \right) dx^{\frac{1}{2}} dy^{\frac{1}{2}}.
\]

\noindent where $R_i$ is a holomorphic function. On the intersection of two open sets, the preimage in $H^0(\mathcal{O}(1,1,0)^+)$ of the difference of the lifts is given by 

\[
(R_i - R_j) \, dx^{\frac{1}{2}} dy^{\frac{1}{2}}.
\]

\noindent Since $dim(H^0(\mathcal{O})) = 1$, the connecting homomorphism $\delta: H^0(\mathcal{O}) \to H^1(\mathcal{O}(1,1,0)^+)$ is either injective or zero. Therefore, it suffices to show that the above element is not in the coboundary.\\

\noindent Suppose to the contrary that it is, then the local sections glue into a global meromorphic section $G(x, y)$ on $C \times C$ with residue $1$. Now fix a nonzero section $v\in H^0(C,L)$ and choose
$y_0\in C$ such that $v(y_0)\neq0$. Then
\[
G(x,y_0)v(x)
\]
is an $L_{y_0}$-valued meromorphic one form on $C$ with a
unique possible pole at $x=y_0$, whose residue is $v(y_0)$.
After choosing a trivialization of $L_{y_0}$, the residue theorem
implies $v(y_0)=0$, which is a contradiction.
\end{proof}

\begin{corollary}
    Let $C$ be such that $H^0(K^{\frac{1}{2}}) = 2$. Then $i: H^0(\mathcal{O}(1,1,0)^+) \to H^0(\mathcal{O}(1,1,1)^+)$ is an isomorphism. Moreover, $ H^0(\mathcal{O}(1,1,0)^+)\cong\mathbb{C}$ and $ H^0(\mathcal{O}(1,1,1)^+)\cong\mathbb{C}$.
\end{corollary}

\begin{proof}
    Consider 
    \[
0 \to \mathcal{O}(1,1,0)^+ \to \mathcal{O}(1,1,1)^+ \to \mathcal{O} \to 0,
\]
which induces the long exact sequence

\[
0 \to H^0(\mathcal{O}(1,1,0)^+) \to H^0(\mathcal{O}(1,1,1)^+) \to H^0(\mathcal{O}) \to H^1(\mathcal{O}(1,1,0)^+) 
\]

\noindent By \cref{lem:injective} the map $\delta: H^0(\mathcal{O}) \to H^1(\mathcal{O}(1,1,0)^+)$ is injective, therefore $ H^0(\mathcal{O}(1,1,0)^+) \to H^0(\mathcal{O}(1,1,1)^+)$ is an isomorphism.\\

\noindent Now since $H^0(\mathcal{O}(1,1,0)^+) \cong \wedge^2 H^0(L)$ by definition and $\dim H^0(L) = 2$, then $H^0(\mathcal{O}(1,1,0)^+)$ is one dimensional.
\end{proof}

\noindent We now consider the following multiplication map:
\begin{equation}\label{mapm}
\begin{array}{rccccl}
m : & \mathrm{Sym}^2 H^0(K) & \otimes & \wedge^2 H^0(K^{1/2}) & \to & \wedge^2 H^0(K^{3/2}) \\
    & a \cdot b & \otimes & \xi \wedge \eta & \mapsto & a\xi \otimes b\eta - a\eta \otimes b\xi + b\xi \otimes a\eta - b\eta \otimes a\xi
\end{array}
\end{equation}
\noindent We proceed to prove that the diagram \cref{commdia} factors through $m$. 
\begin{theorem}\label{thmfactor}
Let $(C_0,L_0)$ be as above and 
\[
\xi:=i^{-1}(S'_0)
\in\wedge^2H^0(L_0).
\]
Then, for every $Q\in I_2(C_0)$,
\[
s'_{1,0}(Q)=m(Q,\xi).
\]
Moreover, $s'_{3,0}=0$.
\end{theorem}

\begin{proof}

\noindent It suffices by linearity to take
\[
Q=a\cdot b\in\operatorname{Sym}^2H^0(K_{C_0})
\qquad\text{and}\qquad
\xi=u\wedge v\in\wedge^2H^0(L_0).
\]
Under the identifications $\operatorname{Sym}^2H^0(K_{C_0})
\cong
H^0\bigl(\mathcal O(2,2,0)^+\bigr)$ and $\wedge^2H^0(L_0)
\cong
H^0\bigl(\mathcal O(1,1,0)^+\bigr)$, these elements correspond respectively to $a\boxtimes b+b\boxtimes a$ and $u\boxtimes v-v\boxtimes u$. Their product is
\[
\begin{aligned}
(a\boxtimes b+b\boxtimes a)
 (u\boxtimes v-v\boxtimes u)
 =
au\boxtimes bv-av\boxtimes bu
+bu\boxtimes av-bv\boxtimes au.
\end{aligned}
\]
Under $H^0\bigl(\mathcal O(3,3,0)^+\bigr)
\cong
\wedge^2H^0(K_{C_0}^{3/2})$, this section corresponds to $au\wedge bv-av\wedge bu$, which is precisely $m(a\cdot b,u\wedge v)$. Hence, for every $Q\in I_2(C_0)$,
\[
s'_{1,0}(Q)
=
Q\cdot S'_0
=
m(Q,\xi).
\]

\noindent Finally, $s'_{3,0}$ is multiplication followed by restriction
to the diagonal, and is therefore multiplication by
$\operatorname{Res}_{\Delta}(S'_0)$. Since $\operatorname{Res}_{\Delta}(S'_0)=0$, we obtain $s'_{3,0}=0$.
\end{proof}

\section{The super-Schottky ideal and powers of the classical Schottky ideal}

\noindent Our goal is to prove the following theorem
\begin{theorem}
    Let d be the minimal number such that $\mathcal{I}_{Sch,g}^d \subset \mathcal{I}_{s-Sch,g}$. Then $d = g$ for all $g\geq 5$. 
\end{theorem}

\noindent We begin by proving a lemma.

\begin{lemma}\label{lem:injectivesym}
Let $C$ be a smooth connected curve such that $H^0(C, L) = 2$ is generated by $\alpha, \beta$. Then for every $m \geq 0$, the multiplication map
$$ \operatorname{Sym}^m H^0(C, L)\longrightarrow H^0(C, L^m) $$
is injective. Equivalently, 
$$ \alpha^m, \alpha^{m-1}\beta, \dots, \alpha\beta^{m-1}, \beta^m $$
are linearly independent.
\end{lemma}

\begin{proof}
Indeed, suppose 
$$ \sum_{i=0}^m c_i \alpha^{m-i}\beta^i = 0. $$
On the nonempty open subset defined by $\beta \neq 0 $, divide by $\beta^m$ to obtain
$$ \sum_{i=0}^m c_i \left(\frac{\alpha}{\beta}\right)^{m-i} = 0. $$
Thus the meromorphic function 
$$ z = \frac{\alpha}{\beta} $$
satisfies a polynomial equation $P(z) = 0$, where 
$$ P(T) = \sum_{i=0}^m c_i T^{m-i}. $$
Because $\alpha$ and $\beta$ are linearly independent, $z$ is nonconstant. But a nonconstant meromorphic function on a smooth curve cannot take values only among the finitely many roots of a nonzero polynomial. Hence $P$ must be the zero polynomial, so 
$$ c_0 = \dots = c_m = 0. $$
\end{proof}

\noindent Now let $C$ be a nonhyperelliptic curve, and let $V=H^{0}(K_{C_{0}})$ with $\dim V=g$, and $L=K_{C_{0}}^{1/2}$ base point free with $H^0(K_{C_0}^{\frac{1}{2}})$ spanned by $\alpha, \beta$. The multiplication map $\operatorname{Sym}^{2}H^{0}(L)\rightarrow V$ defines a 3-dimensional subspace $W\subset V$ spanned by $\alpha^2, \alpha \beta, \beta^2$. Consider
$$0\rightarrow W\rightarrow V\rightarrow V/W\rightarrow 0$$
\noindent which induces 
$$0\rightarrow K_{W}\rightarrow \operatorname{Sym}^{2}V \xrightarrow{r}\operatorname{Sym}^{2}(V/W)\rightarrow 0$$
\noindent where $K_{W}=\operatorname{Im}(W\otimes V\rightarrow \operatorname{Sym}^{2}V)$, then $$\dim K_{W}=\frac{g(g+1)}{2}-\frac{(g-3)(g-2)}{2}=3g-3$$

\noindent Let the canonical ideal $I_{2}(C_{0})$ be the kernel to the multiplication map $\pi:\operatorname{Sym}^{2}V\rightarrow H^{0}(K_{C_{0}}^{\otimes 2})$. We restrict $r$ defined above to this ideal:$$r|_{I_{2}(C_{0})}:I_{2}(C_{0})\rightarrow \operatorname{Sym}^{2}(V/W)$$The kernel of this map is $$\ker(r|_{I_{2}(C_{0})})=I_{2}(C_{0})\cap K_{W}=\ker(\pi|_{K_{W}})$$

\begin{lemma}\label{lem:mu1surjective}
Given $r$, $I_{2}(C_{0})$, $\pi$ and $K_W$ defined above, then    $\dim \ker(\pi|_{K_{W}}) =  \dim \ker(r|_{I_{2}(C_{0})}) = 1$. 
\end{lemma}

\begin{proof}
\noindent We first compute the image $\pi(K_{W})\subset H^{0}(K_{C_{0}}^{\otimes 2})$. Consider the following diagram where the maps are given by multiplications:
$$
\begin{array}{ccccc}
H^{0}(L)^{\otimes 2}\otimes V & \longrightarrow & W\otimes V & \longrightarrow & K_{W} \\
\Big\downarrow id \otimes \mu_1& & & & \Big\downarrow \pi|_{K_{W}} \\
H^{0}(L)\otimes H^{0}(K_{C_{0}}\otimes L) & & \xrightarrow{\quad \mu_{2} \quad} & & H^{0}(K_{C_{0}}^{\otimes 2})
\end{array}
$$

\noindent  Since $L$ is a base point free pencil, we have the short exact sequence 
\begin{equation}\label{basepointfree}
    0\rightarrow L^{-1}\rightarrow H^{0}(L)\otimes\mathcal{O}_{C_{0}}\rightarrow L\rightarrow 0
\end{equation} 
\noindent Tensoring with $K_{C_{0}}$ gives:
\begin{align*}
0\rightarrow L\rightarrow H^{0}(L)\otimes K_{C_{0}}\rightarrow K_{C_{0}}\otimes L\rightarrow 0
\end{align*}
\noindent Taking global sections yields the long exact sequence$$0\rightarrow H^{0}(L)\rightarrow H^{0}(L)\otimes H^{0}(K_{C_{0}})\xrightarrow{\mu_1} H^{0}(K_{C_{0}}\otimes L)\rightarrow H^{1}(L)\rightarrow H^{0}(L)\otimes H^{1}(K_{C_{0}})$$ 

\noindent Note that $\dim ker(\mu_1) = \dim H^{0}(L) = 2$ by the exactness of the sequence, then $rank(\mu_1) = 2g - 2$. On the other hand, $\dim H^{0}( K_{C_{0}} \otimes L)  = 2g-2$ by Riemann Roch. Therefore $\mu_{1}$ is surjective.\\

\noindent Now since all maps are multiplications, the diagram is commutative. Moreover, given that $\mu_{1}$ is surjective, and that the top row is surjective by definition, we obtain  $\pi(K_{W})=\operatorname{Im}(\mu_{2})$. \\

\noindent Consider tensoring the exact sequence \eqref{basepointfree} by $K_{C_{0}}\otimes L$:$$0\rightarrow K_{C_{0}}\rightarrow H^{0}(L)\otimes(K_{C_{0}}\otimes L)\rightarrow K_{C_{0}}^{\otimes 2}\rightarrow 0$$Taking global sections gives:$$0\rightarrow H^{0}(K_{C_{0}})\rightarrow H^{0}(L)\otimes H^{0}(K_{C_{0}}\otimes L)\xrightarrow{\mu_{2}} H^{0}(K_{C_{0}}^{\otimes 2})\xrightarrow{\mu_{3}} H^{1}(K_{C_{0}})\rightarrow H^{0}(L)\otimes H^{1}(K_{C_{0}}\otimes L)$$By Serre duality, $H^{1}(K_{C_{0}}\otimes L)\cong H^{0}(L^{-1})^{\vee}$. Since $\deg(L^{-1})=1-g<0$, then $H^{0}(L^{-1})=0$ and $H^{1}(K_{C_{0}}\otimes L)=0$. Therefore $\operatorname{coker}(\mu_{2}) = H^{0}(K_{C_{0}}^{\otimes 2}) / \operatorname{im}(\mu_{2})  = H^{0}(K_{C_{0}}^{\otimes 2}) / \operatorname{ker}(\mu_{3})  = H^{1}(K_{C_{0}}) \cong\mathbb{C}$, and $\dim \operatorname{Im}(\mu_{2})=3g-4$. Finally $\dim \pi(K_W) = 3g-4$, and

\begin{align*}
    \dim \ker(\pi|_{K_{W}})&=\dim K_{W}-\dim \pi(K_{W})=(3g-3)-(3g-4)=1\\
  \dim \ker(r|_{I_{2}(C_{0})}) &= \dim \ker(\pi|_{K_{W}}) = 1
\end{align*} 

\end{proof}

\noindent Consider the quadric $Q_{null}$ in $\operatorname{Sym}^2 V$ defined by $ \alpha^2 \beta^2 - (\alpha \beta)^2$, which is nonzero since $\alpha^2$, $\beta^2$, $\alpha \beta$ are linearly independent. Moreover it is a degenerate cone containing $C_0$. 

\begin{corollary}
The vector space $\ker(r|_{I_{2}(C_{0})}) $ is spanned by $Q_{null}$.
\end{corollary}

\begin{proof}
    Note $\pi(Q_{null}) = \alpha^2 \beta^2 - \alpha^2 \beta^2 = 0$, therefore it is in $I_2(C_0)$. Furthermore, by definition $Q_{null}$ is in the kernel of $r$. Since $\dim \ker(r|_{I_{2}(C_{0})}) = 1$, then $\ker(r|_{I_{2}(C_{0})}) $ is spanned by $Q_{null}$.
\end{proof}

\noindent We proceed to prove that there is a non-degenerate quadratic form in $\operatorname{Im}(r|_{I_{2}(C_{0})})$.

\begin{theorem}\label{lem:nondeg}
     The image $\operatorname{Im}(r|_{I_{2}(C_{0})})$ must intersect the open dense subset of non-degenerate quadratic forms.
\end{theorem}

\begin{proof}

\noindent Consider $r|_{I_{2}(C_{0})}:I_{2}(C_{0})\rightarrow \operatorname{Sym}^{2}(V/W)$, where$$\dim \operatorname{Im}(r|_{I_{2}(C_{0})})=\frac{(g-2)(g-3)}{2}-1$$Since $\operatorname{Sym}^{2}(V/W)$ has dimension $\frac{(g-2)(g-3)}{2}$, the image of the restricted map is a hyperplane in $\operatorname{Sym}^{2}(V/W)$. The locus of degenerate quadratic forms in $\operatorname{Sym}^{2}(V/W)$ is defined by vanishing determinant, forming an irreducible hypersurface of degree $g-3$. For $g\ge 5$, this degree is $\ge 2$. Since an irreducible hypersurface of degree $d>1$ cannot contain a hyperplane, the image $\operatorname{Im}(r|_{I_{2}(C_{0})})$ must intersect the open dense subset of non-degenerate quadratic forms and it contains at least one nondegenerate quadratic form.
\end{proof}

\noindent Before moving on to prove that the codifferential achieves maximal rank, we set up two lemmas.

\begin{lemma}\label{lem:det} Let $M = \begin{pmatrix}
A & B\\
-B^{T} & D
\end{pmatrix}$ be a skew symmetric matrix where $A$ is invertible, then 
\[
\operatorname{det}
\begin{pmatrix}
A & B\\
-B^{T} & D
\end{pmatrix}
=
\operatorname{det}(A)\,
\operatorname{det}(D+B^{T}A^{-1}B),
\]
    
\end{lemma}

\begin{proof}
Multiplying on the left by
\[
\begin{pmatrix}
I&0\\
B^TA^{-1}&I
\end{pmatrix}
\]
gives
\[
\begin{pmatrix}
A&B\\
0&D+B^TA^{-1}B
\end{pmatrix}.
\]
Note that the multiplying matrix has determinant one, therefore taking
determinants proves the formula.
\end{proof}

\begin{lemma}
\label{lem:general-theta-null}
For $g\geq 5$, let $S_g^+$ be the moduli of even spin curves and let
\[
\Theta_{\mathrm{null}}
:=
\left\{
(C,L)\in S_g^+
:
\dim H^0(C,L)\geq 2
\right\}
\]
be the theta-null divisor in the moduli space of smooth even spin
curves. Then there exists a nonempty Zariski open subset $U\subseteq \Theta_{\mathrm{null}}$ such that for every $(C,L)\in U$, $\dim H^0(C,L)=2$, $|L|\text{ is base-point-free}$, and $C$ is non-hyperelliptic.
\end{lemma}

\begin{proof}
The theta null locus $\Theta_{\mathrm{null}}$ is an irreducible
divisor in $S_g^+$, and its general point $(C,L)$
satisfies $\dim H^0(C,L)=2$ \cite[Remark~2.1 and Theorem~2.4]{Teixidor1988} and \cite[\S1.6]{FarkasIzadi2024}. Furthermore, Farkas--Izadi states that, for a general point
$(C,L)\in\Theta_{\mathrm{null}}$, the pencil $|L|$ induces a
cover
\[
f:C\longrightarrow \mathbb P^1
\]
of degree $g-1$. Since $\deg L=g-1$, this implies that $|L|$
has no fixed divisor and hence is base-point-free. Thus there exists a nonempty Zariski-open subset $U_0\subseteq\Theta_{\mathrm{null}}$ on which $\dim H^0(C,L)=2$ and $|L|$ is base-point-free. Now let $q:S_g^+\longrightarrow\mathcal{M}_g$ map $(C,L)$ to $C$, and $\mathcal H_g\subseteq\mathcal M_g$ be the hyperelliptic locus. The morphism $q$ is finite, therefore
\[
\dim\left(\Theta_{\mathrm{null}}
\cap q^{-1}(\mathcal H_g)\right)
\leq \dim\mathcal H_g=2g-1.
\]
On the other hand $\dim\Theta_{\mathrm{null}}=3g-4$. Then for $g \geq 5$ we have $2g-1<3g-4$, and $U_{1}
:=
\Theta_{\mathrm{null}}\setminus q^{-1}(\mathcal H_g)$ is a nonempty Zariski-open subset. Since
$\Theta_{\mathrm{null}}$ is irreducible,
\[
U:=U_0\cap U_{1}
\]
is nonempty and satisfies all the required properties.
\end{proof}

\noindent Before moving on the final theorem we explicitly calculate the codifferential of the superperiod map at $C_0$.

\begin{lemma}
After choosing the basis $\alpha,\beta$ so that $i^{-1}(S'_0)=\alpha\wedge\beta$, one has
\[
s'_{1,0}(Q_{\mathrm{null}})
=
\alpha^3\wedge\beta^3
-
3\alpha^2\beta\wedge\alpha\beta^2
\]

\end{lemma}

\begin{proof}
\begin{align*}
s'_{1,0}(Q_{null}) &= m(Q_{null}, i^{-1}(S_0')) \\
         &= m(\alpha^2 \cdot \beta^2 - \alpha\beta \cdot \alpha\beta, \alpha \wedge \beta) \\
         &= \alpha^3 \wedge \beta^3 - 3\alpha^2\beta \wedge \alpha\beta^2
\end{align*}    
\end{proof}
\noindent Now we proceed to the final theorem

\begin{theorem}\label{theorem}
Let \(g\geq 5\), and let \((C_{0},L_0)\) be a generic point of the vanishing theta-null locus, where
\(C_{0}\) is a non-hyperelliptic curve of genus \(g\), \(L_0^{\otimes 2}\cong K_{C_{0}}\), \(H^{0}(L_0)=2\), and \(|L_0|\) is base-point-free. Then there exists a quadric
\[
Q_{0}\in I_{2}(C_{0})
\]
such that the rescaled limiting codifferential
\[
s'_{1,0}(Q_{0})\in \bigwedge^{2}H^{0}(K_{C_{0}}^{3/2})
\]
has maximal rank \(2g-2\). Consequently, for a generic deformation \((C_{t},L_t,Q_{t})\) where $C_t$ is non-hyperelliptic specializing to \((C_{0},L_0,Q_{0})\), with
\[
Q_{t}\in I_{2}(C_{t}),
\]
the ordinary codifferential
\[
s_{1,t}(Q_{t})
\]
has maximal rank \(2g-2\) for all sufficiently small \(t\neq 0\).

\end{theorem}

\begin{proof}

The subspace \(W\subset V\) induces a filtration
\[
F^{0}\subset F^{1}:=H^{0}(K_{C_0}^{3/2}),
\]
where $F^{0}:=\operatorname{Im}\left(\operatorname{Sym}^{3}H^{0}(L_0)
\longrightarrow H^{0}(K_{C_0}^{3/2})\right)$.
Let \(|L_0|\) be generated by \(\alpha,\beta\), then the map
\[
\operatorname{Sym}^{3}H^{0}(L_0)\longrightarrow H^{0}(L_0^{3})
=
H^{0}(K_{C_0}^{3/2})
\]
is injective by \cref{lem:injectivesym}. Therefore the four monomials $\alpha^{3}, \alpha^{2}\beta,\alpha\beta^{2}, \beta^{3}$ are linearly independent, and $\dim F^{0}=4$.\\

\noindent Consider $\mu_{1}:V\otimes H^{0}(L_0)\longrightarrow H^{0}(K_{C_0}\otimes L_0)
=
H^{0}(K_{C_0}^{3/2})$, which is surjective from \cref{lem:mu1surjective}. Moreover, the image of \(W\otimes H^{0}(L_0)\) under $\mu_1$ is $F^0$. Therefore \(\mu_{1}\) descends to a surjective map
\[
(V/W)\otimes H^{0}(L_0)
\twoheadrightarrow
H^{0}(K_{C_0}^{3/2})/F^{0}.
\]

\noindent Now $\dim((V/W)\otimes H^{0}(L_0)) = 2(g-3) = 2g-6$. On the other hand, $\dim H^{0}(K_{C_0}^{3/2})=2g-2$ by Riemann--Roch and $\dim(H^{0}(K_{C_0}^{3/2})/F^{0}) = (2g-2)-4 = 2g-6$.
Thus the induced map is an isomorphism, and $F^{1}/F^{0}
\cong (V/W)\otimes H^{0}(L_0)$.\\

\noindent Choose lifts \(a_{1},\dots,a_{g-3} \in V\) of a basis of \(V/W\), and then a basis of \(F^{1}\) by first taking $\alpha^{3}, \alpha^{2}\beta, \alpha\beta^{2},\beta^{3}$ for \(F^{0}\), followed by elements projecting to the quotient $a_{1}\alpha,\dots,a_{g-3}\alpha,
a_{1}\beta,\dots,a_{g-3}\beta$. By \cref{lem:nondeg}, we may choose \(\widetilde Q\in I_{2}(C_0)\) whose projection to
\(\operatorname{Sym}^{2}(V/W)\) is a non-degenerate quadratic form 

\begin{align*}
    Q_U = \sum_{j,k}c_{jk}a_{j}a_{k}
\end{align*}

\noindent Now consider the family of quadrics
\[
Q(\epsilon)=Q_{\mathrm{null}}+\epsilon \widetilde Q.
\]

\noindent With respect to the basis $\alpha^{3}, \alpha^{2}\beta, \alpha\beta^{2},\beta^{3}, a_{1}\alpha,\dots,a_{g-3}\alpha,
a_{1}\beta,\dots,a_{g-3}\beta$, the image of $Q_{\mathrm{null}}$ under the rescaled codifferential
\(s'_{1,0}(Q_{\mathrm{null}})\) is a $2g-2$ by $2g-2$ matrix given by 

\[s'_{1,0}(Q_{\mathrm{null}})
=\begin{pmatrix}
M_{null} & 0  \\
0 & 0 
\end{pmatrix}
\text{ where } M_{null}
=
\begin{pmatrix}
0 & 0 & 0 & 1 \\
0 & 0 & -3 & 0 \\
0 & 3 & 0 & 0\\
-1 & 0 & 0 & 0  \\
\end{pmatrix}
\]

\noindent  To compute \(s'_{1,0}(\epsilon \widetilde Q)\), by \cref{thmfactor} and \cref{mapm} consider

\[
m(a_{j}a_{k},\alpha\wedge\beta)
=
a_{j}\alpha\wedge a_{k}\beta
+
a_{k}\alpha\wedge a_{j}\beta.
\]
Therefore, with respect to the same basis, \(s'_{1,0}( \widetilde Q)\) is given by

\[
\begin{pmatrix}
A & B\\
-B^{T} & D
\end{pmatrix},
\text{ where } D
=
\begin{pmatrix}
0 & C\\
-C^T & 0
\end{pmatrix},
\]

\noindent where $A$ is an antisymmetric $4$ by $4$ matrix and $D$ is the $2g-6$ by $2g-6$ antisymmetric matrix representing \(Q_{U}\). Since \(Q_{U}\) is non-degenerate, then $\det(C)\neq0$, $\det(D)=\det(C)^{2}\neq0$ and \(D\) has full rank \(2g-6\).\\

\noindent Since $s'_{1,0}$ is linear, 

\[
s'_{1,0}(Q(\epsilon)) = \begin{pmatrix}
M_{null} & 0  \\
0 & 0 
\end{pmatrix}+\epsilon \begin{pmatrix}
A & B\\
-B^{T} & D
\end{pmatrix} = \begin{pmatrix}
M_{null}  + \epsilon A &\epsilon  B\\
-\epsilon B^{T} & \epsilon D
\end{pmatrix}
\]\\

\noindent Now $M_{null} $ has full rank, therefore for sufficiently small nonzero $\epsilon$ we have that $M_{null}  + \epsilon A$ is invertible. By \cref{lem:det} we have 

\begin{align*}
\operatorname{det}(s'_{1,0}(Q(\epsilon)) &= \operatorname{det}(M_{null}+\epsilon A)\operatorname{det}(\epsilon D+\epsilon B^T (M_{null}+\epsilon A)^{-1} \epsilon B) \\
&= \epsilon^{2g-6}\operatorname{det}(M_{null}+\epsilon A)\operatorname{det}( D+\epsilon B^T (M_{null}+\epsilon A)^{-1}  B)
\end{align*}

\noindent Expanding the above we obtain
\[
\operatorname{det}(s'_{1,0}(Q(\epsilon)))
=
\epsilon^{2g-6}
\operatorname{det}(M_{\mathrm{null}})
\operatorname{det}(D)
+
\mathcal O(\epsilon^{2g-5}).
\]

\noindent Since $M_{\mathrm{null}}$ and $D$ are non-degenerate, we have $\operatorname{det}(M_{\mathrm{null}})\neq0$ and  $\operatorname{det}(D)\neq0$. Therefore the leading coefficient is nonzero, and $\operatorname{det}(s'_{1,0}(Q(\epsilon)))$ is nonzero for all sufficiently small nonzero $\epsilon$.\\

\noindent Finally, over the non-hyperelliptic locus the multiplication map
\[
\operatorname{Sym}^2 H^0(K) \to H^0(K^2)
\] is surjective, and the kernel forms a locally free bundle. Fix $Q_0 = Q_{null} + \epsilon \tilde{Q}$ for sufficiently small nonzero $\epsilon$ such that $\operatorname{det}(s'_{1,0}(Q_0))$ is nonzero, and a one-parameter family of spin curves $(C_t, L_t)$ transverse to the theta null divisor, specializing to $(C_0,L_0)$. Since the non-hyperelliptic locus is open and $C_0$ is
non-hyperelliptic, we may
assume that $C_t$ is non-hyperelliptic for sufficiently
small $t$. Now, as the kernel forms a locally free bundle, $Q_0$ extends to $Q_t \in I_2 (C_t)$. Now $s'_{1,t}(Q_t) \to s'_{1,0} (Q_0)$ holomorphically and $\operatorname{det}(s'_{1,0}(Q_0)) \neq 0$, therefore $\operatorname{det}(s'_{1,t}(Q_t)) \neq 0$ for sufficiently small $t$. Since multiplication by the non-zero coordinate \(t\) does not change rank, then
\[
s_{1,t}(Q_{t})
\]
has full rank \(2g-2\) for sufficiently small \(t\neq0\).

\end{proof}

\begin{theorem}
    Let d be the minimal number such that $\mathcal{I}_{Sch,g}^d \subset \mathcal{I}_{s-Sch,g}$. Then $d = g$ for all $g\geq 5$. 
\end{theorem}

\begin{proof}
\noindent Let $t \neq 0$ be such that $s_{1,t}(Q_t)$ has full rank and choose local odd coordinates
$ \eta_1, \dots, \eta_{2g-2}$. Let $f \in I_{\text{Sch}}$ be a local function whose conormal class is $Q_t$. Then
$$ \text{per}^* f = \frac{1}{2} \sum_{i,j} A_{ij}\eta_i\eta_j + O(\eta^4), $$
where $A$ is the skew matrix representing $s_{1,t}(Q_t)$. Consider $(\text{per}^* f)^{g-1}$, then any term which contains at least one factor of degree at least four has degree at least $4+2(g-2)=2g$. Since there are only $2g-2$ odd variables, every term of degree
at least $2g$ vanishes. Consequently, $ (\text{per}^* f)^{g-1} = c_g \text{Pf}(A) \eta_1 \cdots \eta_{2g-2} $
for some nonzero numerical constant $c_g$. Now since $A$ has full rank,
$ \text{Pf}(A) \neq 0, $
and therefore
$\text{per}^*(f^{g-1})  = \text{per}^*(f)^{g-1}\neq 0. $ 
Thus $f^{g-1}\notin (I_{s-\mathrm{Sch},g})_x$ at the corresponding period point $x$. Hence $I_{\mathrm{Sch},g}^{\,g-1}
\not\subseteq I_{s-\mathrm{Sch},g}$.\\

\noindent For the opposite inclusion, let $\mathcal{J}$ be the ideal generated by the $2g - 2$ odd variables. Since a classical Schottky function vanishes on the reduced period image,
$ \text{per}^* I_{\text{Sch}} \subseteq \mathcal{J}^2.$ Therefore
$ \text{per}^*(I_{\text{Sch}}^g) \subseteq \mathcal{J}^{2g} = 0, $ and
$ I_{\text{Sch}}^g \subseteq I_{s-\text{Sch}}. $

\end{proof}

\section{Acknowledgements}

\noindent The author would like to thank Professor Ron Donagi for proposing the problem and his valuable feedback.

\bibliographystyle{amsplain}  
\bibliography{references} 

@article{FelderKazhdanPolishchuk2019,
  author = {Giovanni Felder and David Kazhdan and Alexander Polishchuk},
  title = {Regularity of the Superstring Supermeasure and the Superperiod Map},
  journal = {arXiv preprint arXiv:1905.12805},
  year = {2019},
  url = {https://arxiv.org/abs/1905.12805}
}

@article{Witten2012superRiemann,
  author    = {Edward Witten},
  title     = {Notes on Super Riemann Surfaces and Their Moduli},
  journal   = {arXiv preprint arXiv:1209.2459},
  year      = {2012},
  url       = {https://arxiv.org/abs/1209.2459}
}

@article{DonagiWitten2013,
  author  = {Donagi, Ron and Witten, Edward},
  title   = {Super Atiyah Classes and Obstructions to Splitting of Supermoduli Space},
  journal = {Pure and Applied Mathematics Quarterly},
  year    = {2013},
  volume  = {9},
  pages   = {739--788},
  doi     = {10.4310/pamq.2013.v9.n4.a5}
}

@article{DHoker1989,
  author  = {D'Hoker, Eric and Phong, D. H.},
  title   = {Conformal scalar fields and chiral splitting on super Riemann surfaces},
  journal = {Communications in Mathematical Physics},
  year    = {1989},
  volume  = {125},
  pages   = {469--513},
  doi     = {10.1007/bf01218413}
}

@article{Teixidor1988,
  author    = {Teixidor i Bigas, Montserrat},
  title     = {The divisor of curves with a vanishing theta-null},
  journal   = {Compositio Mathematica},
  volume    = {66},
  number    = {1},
  pages     = {15--22},
  year      = {1988},
  mrnumber  = {937985},
  zbl       = {0663.14017},
  url       = {https://www.numdam.org/item/CM_1988__66_1_15_0/}
}

@misc{FarkasIzadi2024,
  author        = {Farkas, Gavril and Izadi, Elham},
  title         = {Szeg{\H{o}} kernels and Scorza quartics on the moduli space of spin curves},
  year          = {2024},
  eprint        = {2409.13303},
  archivePrefix = {arXiv},
  primaryClass  = {math.AG},
  doi           = {10.48550/arXiv.2409.13303}
}

@article{DonagiNoja-genus4,
  author = {Donagi, Ron and Noja, Simone},
  title  = {{Schottky} versus super {Schottky} in genus 4},
  note   = {Companion paper},
  year   = {2026},
}
\end{document}